\documentclass{amsart}
\usepackage{amsmath,amssymb,amsthm,url,xcolor}
\newtheorem{thr}{Theorem}
\newtheorem{lem}[thr]{Lemma}
\newtheorem{conj}[thr]{Conjecture}
\newtheorem{cor}[thr]{Corollary}

\newtheorem{claim}[thr]{Claim}

\theoremstyle{definition}
\newtheorem{defn}[thr]{Definition}

\theoremstyle{remark}

\numberwithin{equation}{section}

\def\A{\mathcal{A}}
\def\N{\mathbb{N}}

\def\B{\mathcal{B}}

\def\F{\mathbb{F}}
\def\Mat{\mathrm{Mat}}

\begin{document}

\title[An improved bound for the length of matrix algebras]{An improved bound for the \\ length of matrix algebras}

\author{Yaroslav Shitov}
\email{yaroslav-shitov@yandex.ru}







\begin{abstract}
Let $S$ be a set of $n\times n$ matrices over a field $\F$. 
We show that the $\F$-linear span of the words in $S$ of length at most
$$2n\log_2n+4n$$
is the full $\F$-algebra generated by $S$.
This improves on the $n^2/3+2/3$ bound by Paz (1984) and an $O(n^{3/2})$ bound of Pappacena (1997).
\end{abstract}

\maketitle

Let $S$ be a subset of a finite-dimensional associative algebra $\mathcal{A}$ over a field $\F$. An element $a\in\A$ is said to be a \textit{word} of length $k$ in $S$ if there are $a_1,\ldots,a_k\in S$ such that $a=a_1\ldots a_k$. We denote the set of all such words by $S^k$, and we write $\F S^k$ for the $\F$-linear span of $S^k$. Similarly, $\F S^{\leqslant k}$ will stand for the $\F$-linear span of all the words in $S$ that have length at most $k$.

\begin{defn}
The length $\ell(S)$ is the smallest integer $k$ for which $\F S^{\leqslant k}$ is the full subalgebra generated by $S$. We also define $\ell(\mathcal{A})$ as the maximum value of $\ell(S)$, where $S$ runs over all subsets of $\mathcal{A}$ that generate $\mathcal{A}$ as an $\F$-algebra.
\end{defn}

In our paper, we study the length of $\Mat_n(\F)$, the set of $n\times n$ matrices viewed as an algebra over $\F$.
In 1984, A. Paz~\cite{Paz} proved that $\ell(S)\leqslant n^2/3+2/3$ for all $S\subset \Mat_n(\F)$ and proposed the following appealing conjecture.

\begin{conj}\label{conjPaz}
For all $S\subset\Mat_n(\F)$, one has $\ell(S)\leqslant 2n-2$.
\end{conj}

As shown by T. Laffey in~\cite[p. 131]{TL}, the upper bound in Conjecture~\ref{conjPaz} should be sharp. This conjecture is known to hold if the size of matrices is at most four~\cite{Paz} or if $\F S$ contains a non-derogatory matrix~\cite{GLMS}. However, the best known general upper bounds on the lengths of matrix subsets are quite far from the one prescribed by Conjecture~\ref{conjPaz}. It was only in 1997 when a subquadratic estimate was obtained: C. Pappacena proved an $O(n^{3/2})$ upper bound on the length of $\Mat_n(\F)$, but no further improvements have been made since then~\cite{GLMS, Au1, Au2}. The main result of this paper is a much stronger $O(n \log n)$ upper bound on the length of $\Mat_n(\F)$.

\begin{thr}\label{thrlen}
For all $S\subset\Mat_n(\F)$, we have $\ell(S)\leqslant 2n\log_2n+4n-4$.
\end{thr}

As an additional motivation of our study, we note that the best known upper bounds on a complete set of unitary invariants for $n\times n$ matrices~\cite{TL} and on the PI degree of semiprime affine algebras of Gelfand--Kirillov dimension one~\cite{PappGL} come from the estimates of $\ell(\Mat_n(\F))$,
so the current work does also improve our understanding of those invariants.

\section{Warm-up}

In this section, we explain the idea behind our main construction and illustrate its work in a simpler setting. We get a small improvement on one of the results of Pappacena's~\cite{Papp}, which allows us to prove the $n=5$ case of Conjecture~\ref{conjPaz}.

We say that a set $S\subset\Mat_n(\F)$ is \textit{irreducible} if it generates $\Mat_n(\F)$ as the $\F$-algebra. If a set $S$ is not irreducible, and if $\F$ is algebraically closed, then there exist $p\in\{1,\ldots,n-1\}$ and $Q\in\operatorname{GL}_n(\F)$ such that, for any $A\in S$, we have
\begin{equation}\label{eq1}Q^{-1} A Q=\left(\begin{array}{c|c}
A_{11}&A_{12}\\\hline
O&A_{22}
\end{array}\right)\end{equation}
with $A_{11}$ being a $p\times p$ matrix. (This is \textit{Burnside's theorem}, see~\cite[Theorem~1.5.1]{RRBook}.)

\begin{lem}{\upshape(See Corollary~3 in~\cite{Mar}.)}\label{lemOV}
Let $\mathcal{A}$ be a matrix algebra whose elements are of the form~\eqref{eq1}, and let $\mathcal{A}_1,\mathcal{A}_2$ be the sets of all $A_{11}, A_{22}$ blocks of matrices in $\mathcal{A}$, respectively. Then $\ell(\mathcal{A})\leqslant\ell(\mathcal{A}_1)+\ell(\mathcal{A}_2)+1$.
\end{lem}

We will say that a matrix $Z\in\Mat_n(F)$ is \textit{square-zero} if $Z^2=0$. The main idea of the proof of Theorem~\ref{thrlen} is to control the product $\lambda\rho(\lambda)$, where $\rho(\lambda)$ is the minimal rank of non-zero square-zero matrices that arise as linear combinations of words of length at most $\lambda$. We show in Section~3 below that we can reduce $\rho$ to $1$ whilst saving the property $\lambda\rho(\lambda)\in O(n\log n)$, and then we apply Pappacena's technique to deal with low rank matrices, see~\cite[Theorem~4.1]{Papp} and Corollary~\ref{lemPap} below. More precisely, let $H\in\F S^{\leqslant \lambda}$ be a square-zero matrix; it can be written as
$$H=\left(\begin{array}{c|c|c}
O&O&I_\rho\\\hline
O&O&O\\\hline
O&O&O
\end{array}\right)$$
with respect to some basis. If some matrix $A$ with bottom-left block of small rank $r>0$ comes as a linear combination of words of length $l$, then the matrix $HAH$ is square-zero, has rank $r$, and comes as a linear combination of words of length at most $l+2\lambda$. As we will see in Claims~\ref{lem4} and~\ref{lem5} below, we can always find an appropriate matrix $A$ to reduce the rank of a square-zero matrix. The following lemma illustrates our approach to the proof of Claim~\ref{lem4}.

\begin{lem}\label{lemill}
Consider an irreducible set $S\subset\F^{n\times n}$ and a non-zero vector $v\in\F^{n}$. If $\F S^{\leqslant(n-2)}v\neq\F^n$, then $\F S$ contains a matrix with minimal polynomial of degree $n$.
\end{lem}

\begin{proof}
The sequence $$\F v=\F S^0v\subset \F S^{\leqslant 1}v\subset\ldots\subset \F S^{\leqslant k} v=\F^n$$ is strictly increasing~\cite[Theorem~4.1]{Papp}, so the assumption of the lemma implies $k=n-1$ and $\dim \F S^{\leqslant t}v-\dim \F S^{\leqslant (t-1)}v=1$ for all $t\in\{1,\ldots,n-1\}$. Therefore, we can set $\B_0=\{v\}$ and inductively complete $\B_{t-1}$ to a basis $\B_t$ of $\F S^{\leqslant t}$ by adding a single vector $v_{t}$. With respect to the basis $\{v,v_1,\ldots,v_{n-1}\}$, every matrix in $S$ has the form
$$A=\begin{pmatrix}
* & \ldots &\ldots&*&*\\
a_{21} & * &\ldots&*&*\\
0 & \ddots &\ddots&\vdots&\vdots\\
\vdots & \ddots &\ddots&*&*\\
0 & \ldots&0 &a_{n,n-1}&*\\
\end{pmatrix}$$
with $*$'s denoting the entries we need not specify. Since $S$ is irreducible, every of the $(i+1,i)$ entries is non-zero at some matrix in $S$, so a generic element of $\F S$ has all of them non-zero --- which means that its minimal polynomial has degree $n$.
\end{proof}

\begin{thr}\label{thrnond}{\upshape (Theorems~2.4 and~2.5 in~\cite{GLMS}.)}
If an irreducible set $\F S\subset\Mat_n(\F)$ contains a matrix with minimal polynomial of degree $n-1$ or $n$, then $\ell(S)\leqslant 2n-2$.
\end{thr}

Lemma~\ref{lemill} and Theorem~\ref{thrnond} lead to a tiny improvement of the $r=1$ case of Theorem~4.1(a) in~\cite{Papp}, which is nevertherless useful to study the case of small $n$.

\begin{cor}\label{lemPap}
Let $S\subset\Mat_n(\F)$ be an irreducible set and $k\geqslant 2$. If $\,\F S^{\leqslant k}$ contains a rank-one matrix, then $\ell(S)\leqslant 2n+k-4$.
\end{cor}

\begin{proof}
If $\F S$ contains a matrix with minimal polynomial of degree $n$, then we are done by Theorem~\ref{thrnond}. Otherwise, we use Lemma~\ref{lemill} and get
$$\F S^{\leqslant(n-2)} A S^{\leqslant(n-2)} =\sum\Mat_n(\F)\cdot A\cdot \Mat_n(\F)=\Mat_n(\F)$$
for any rank-one matrix $A$.
\end{proof}

We are almost ready to prove the $n=5$ case of Conjecture~\ref{conjPaz}.

\begin{claim}\label{cldeg2}
Assume that the minimal polynomial of every matrix in $\F S\subset\Mat_n(\F)$ has degree at most $2$. Then $\ell(S)\leqslant2\log_2 n$.
\end{claim}

\begin{proof}
We denote by $w$ a word in $S^{\ell(S)}$ that is not spanned by shorter words. For any $A,B\in S$, the matrices $A^2$ and $AB+BA=(A+B)^2-A^2-B^2$ belong to $\F S^{\leqslant 1}$, which implies that the letters of $w$ are all different and their permutations do not break the property of $w$ not to be spanned by shorter words. In particular, the products corresponding to the different $2^{\ell(S)}$ subsets of letters of $w$ should be linearly independent, which implies $2^{\ell(S)}\leqslant\dim\Mat_n(\F)$.
\end{proof}

\begin{thr}\label{thr5}
If $S\subset\Mat_5(\F)$, then $\ell(S)\leqslant 8$.
\end{thr}

\begin{proof}
Since a set of vectors is linearly dependent over $\F$ if it is linearly dependent over the algebraic closure of $\F$, it is sufficient to prove the statement assuming that $\F$ is algebraically closed~\cite[page~239]{GLMS}. Moreover, Conjecture~\ref{conjPaz} is known to hold for $n\leqslant4$ (see~\cite{Paz}), so we can use Lemma~\ref{lemOV} and assume without loss of generality that $S$ is irreducible. According to Theorem~\ref{thrnond} and Claim~\ref{cldeg2}, we can restrict to the case when $\F S$ contains a matrix $A$ with minimal polynomial of degree $3$. A straightforward analysis of possible Jordan forms of $A$ shows that the linear span of $I, A, A^2$ must contain a rank-one matrix, so it remains to apply Corollary~\ref{lemPap}.
\end{proof}

As said above, the case of $n\leqslant 4$ in Conjecture~\ref{conjPaz} was considered in 1984 by Paz~\cite{Paz}, but the case of $n=5$ remained open until now~\cite{GLMS}. Let us mention the works~\cite{Au1, Au2}, which cover the case $n\leqslant 6$ under the additional assumption of $\dim\F S\leqslant 2$.

\section{The proof of Theorem~\ref{thrlen}}

Let $A$ be an $n\times n$ matrix over a field $\F$, which is assumed to be algebraically closed in this section. We recall that there exists 
$Q\in\operatorname{GL}_n(\F)$ such that $Q^{-1}AQ$ has \textit{rational normal form}, that is, we have $Q^{-1}AQ=\operatorname{diag}(C_{f_1},\ldots,C_{f_k})$, where
$$C_f=\begin{pmatrix}
0 & 0 &\ldots&0&-c_0\\
1 & 0 &\ldots&0&-c_1\\
0 & 1 &\ldots&0&-c_2\\
\vdots & \vdots &\ddots&\vdots&\vdots\\
0 & 0 &\ldots&1&-c_{m-1}\\
\end{pmatrix}$$
is the companion matrix of a polynomial $f=t^m+c_{m-1} t^{m-1}+\ldots+c_0$, and the \textit{invariant factors} $f_1,\ldots,f_k$ satisfy $f_1|\ldots|f_k$.

\begin{claim}\label{lem1}
Let $\delta$ be the degree of the minimal polynomial of an $n\times n$ matrix $A$ over $\F$. Then the $\F$-linear span of $I,A,\ldots,A^{\delta-1}$ contains either a non-zero projector of rank at most $n/\delta$ or a non-zero square-zero matrix of rank at most $n/\delta$.
\end{claim}

\begin{proof}
Let $\psi$ be a polynomial that has degree $\delta-1$, divides the minimal polynomial $\varphi$ of $A$, and is a multiple of any other invariant factor of $A$. Then $\psi(A)$ has equal rank-one matrices in the places of the largest blocks of the rational normal form of $A$ and zeros everywhere else.
\end{proof}

\begin{claim}\label{lem3}
For any irreducible set $S\subset\Mat_n(\F)$, there exist non-zero $\lambda,\rho$ such that $\lambda\rho\leqslant 2n$ and $\F S^{\leqslant\lambda}$ contains a square-zero matrix of rank $\rho$.
\end{claim}

\begin{proof}
We apply Claim~\ref{lem1} to any non-scalar matrix in $S$ and find a non-zero matrix $P\in \F S^{\leqslant(\delta-1)}$ that has rank at most $n/\delta$ and satisfies either $P^2=P$ or $P^2=0$. We are done if $P^2=0$; otherwise $H_B=(I-P)BP$ is a square-zero matrix for all $B$. We can have $H_B=0$ only when $\operatorname{Im} P$ is invariant with respect to $B$, but since $S$ is irreducible, this obstruction cannot happen for all $B\in S$.
\end{proof}

\begin{claim}\label{lem2}
Let $A\in\F^{n\times n}$ and $r\in\N$. Assume that $\operatorname{rank}(PAQ)\leqslant r$ holds for all $P\in\F^{p\times n}$, $Q\in\F^{n\times q}$ satisfying $PQ=0$. Then $\operatorname{rank}(A-\mu I)\leqslant 2r$ for some $\mu\in\F$.
\end{claim}

\begin{proof}
Both the assumption and conclusion are independent of the substitution $A\to C^{-1}AC$, so we can assume that $A$ has rational normal form. We denote the number of diagonal blocks by $k$ and their sizes by $m_1,\ldots,m_k$. We have $\min_\mu\operatorname{rank}(A-\mu I)=n-k$, and we are going to conclude the proof by constructing a unit square submatrix $A'=A[I|J]$ with $I\cap J=\varnothing$ and $|I|=|J|\geqslant 0.5(n-k)$. Namely, we pick a family of $\lfloor m_t/2\rfloor$ non-consecutive sub-diagonal ones from a $t$th diagonal block of $A$, and the union of all such families will be the diagonal of $A'$.
\end{proof}

\begin{claim}\label{lem4}
Let $S\subset\F^{n\times n}$, $P\in\F^{p\times n}$, $Q\in\F^{n\times q}$. Let $k$ be the smallest integer such that $PS^kQ\neq0$. Then, for any $A_1,\ldots,A_k\in S$, we have $\operatorname{rank}(PA_1\ldots A_kQ)\leqslant n/k$.\end{claim}

\begin{proof}
Let $V_0=\operatorname{Im}Q$ and $V_t=\sum_{M\in S^{\leqslant t}}\operatorname{Im}MQ$. Let $\B_0,\ldots,\B_k\subset\F^n$ be vector families such that $\B_0\cup\ldots\cup\B_t$ is a basis of $V_t$ for $t=0,\ldots,k$. Let $\mathcal{C}\subset\F^n$ be such that $\B_0\cup\ldots\cup\B_k\cup\mathcal{C}$ is a basis of $\F^n$. Every matrix $A\in S$ has the form
$$\left(\begin{array}{c|c|c|c|c|c|c}
 & \B_0 &\B_1 &\ldots &\B_{k-1} &\B_k&\mathcal{C} \\\hline
\B_0 & *  &\ldots &\ldots &\ldots &*&* \\\hline
\B_1 & A(1,0)& *  &\ldots &\ldots  &*&* \\\hline
\B_2 & O & A(2,1)& *   &\ldots  &*&* \\\hline
\vdots&\vdots&O&\ddots&*&\vdots&\vdots\\\hline
\B_k&\vdots&\vdots&\ddots&A(k,k-1)&*&*\\\hline
\mathcal{C}&O&O&\ldots&O&*&*
\end{array}\right),$$
where the $*$'s stand for entries that we need not specify, and the left column and top row of the matrix above indicate the basis vectors the respective blocks of rows and columns correspond to. We also have $P=\left(O|\ldots|O|P'|*\right)$, $Q=(Q'|O|\ldots|O)^\top$ with some matrices $P',Q'$ at the $\B_k$ position of $P$ and the $\B_0$ position of $Q$, respectively. For $A_1,\ldots,A_k\in S$, the matrix $PA_k\ldots A_1Q$ equals $P'A_k(k,k-1)\ldots A_1(1,0)Q'$, so its rank is at most $\min_t|\B_t|\leqslant n/k$.
\end{proof}

\begin{claim}\label{lem5}
Let $S\subset\Mat_n(\F)$ be an irreducible set, and assume that $\F S^{\leqslant \lambda}$ contains a square-zero matrix $H$ of rank $\rho\geqslant 2$. Then there exist $\rho_1\in\left[1,0.5{\rho}\right]$ and
$$\lambda_1\leqslant \frac{\lambda\rho}{\rho_1}+\frac{4n(\rho-\rho_1)}{\rho\rho_1}$$
such that $\F S^{\leqslant \lambda_1}$ contains a square-zero matrix of rank equal to $\rho_1$.
\end{claim}

\begin{proof}
Let $P\in\F^{p\times \rho}$, $Q\in\F^{\rho\times q}$ be non-zero matrices satisfying $PQ=0$. We choose a basis such that
$$H=\left(\begin{array}{c|c|c}
O&O&I_\rho\\\hline
O&O&O\\\hline
O&O&O
\end{array}\right)$$
and define $P'=(O|O|P)$ and $Q'=(Q|O|O)^\top$. Let $k$ be the smallest integer for which there exist $P',Q'$ defined as above and $A_1,\ldots,A_k\in S$ satisfying $P'A_1\ldots A_kQ'\neq0$ (such an integer exists because $S$ is irreducible). We write $A=A_1\ldots A_k$, and we denote by $A'$ the bottom left block of $A$. Since $PA'Q\neq0$, the matrix $A'$ is non-scalar, that is, its minimal polynomial has degree $\delta>1$.

\textit{Case 1.} Assume $k\leqslant 4n/\rho$. By Claim~\ref{lem1}, there is a polynomial $\psi$ of degree at most $(\delta-1)$ such that $\rho_1:=\operatorname{rank}\psi(A')\in[1, \rho/\delta]$; we see that $H_1=\psi(HA)H$ is a square-zero matrix of rank $\rho_1$. It remains to note that $H_1$ is spanned by words of length at most $(\delta-1)(\lambda+k)+\lambda \leqslant\lambda\delta+(\delta-1)k\leqslant \lambda\rho/\rho_1+4n(\rho/\rho_1-1)/\rho$.

\textit{Case 2.} Now let $k\geqslant 4n/\rho$. The matrix $HAH$ has $A'$ at the upper right block and zeros everywhere else. According to Claim~\ref{lem4}, we have $\operatorname{rank}(PA'Q)\leqslant n/k$ for any choice of $P,Q$ as above. 
We set $H_1=HAH-\mu H$ with $\mu\in\F$, and we conclude by Claim~\ref{lem2} that $\rho_1:=\operatorname{rank}(H_1)\leqslant 2n/k$.
So we have $\rho_1\leqslant 0.5\rho$, and $H_1$ is spanned by words of length at most $2\lambda+k\leqslant \lambda\rho/\rho_1+2n/\rho_1\leqslant\lambda\rho/\rho_1+4n(1-\rho_1/\rho)/\rho_1$.
\end{proof}


%

\begin{proof}[Proof of Theorem~\ref{thrlen}.]
As in the proof of Theorem~\ref{thr5}, we can assume without loss of generality that $\F$ is algebraically closed and $S$ is irreducible. Using Claim~\ref{lem3}, we find a square-zero matrix of rank $\rho_0>0$ in $\F S^{\leqslant\lambda_0}$ with $\lambda_0\rho_0\leqslant 2n$; if $\rho_0=1$, then we apply Corollary~\ref{lemPap} and complete the proof. Otherwise, we repeatedly apply Claim~\ref{lem5} and obtain a sequence $(\lambda_0,\rho_0),\ldots,(\lambda_\tau,\rho_\tau)$ such that $\rho_\tau=1$ and for all $t\in\{0,\ldots,\tau-1\}$ it holds that $\rho_{t+1}\in\left[1,0.5\rho_{t}\right]$,
$$\lambda_{t+1}\leqslant \frac{\lambda_t\rho_t}{\rho_{t+1}}+\frac{4n(\rho_t-\rho_{t+1})}{\rho_t\rho_{t+1}},$$
and every $\F S^{\leqslant \lambda_t}$ contains a square-zero matrix of rank $\rho_t$. By induction we get
$$\lambda_t\leqslant \frac{\lambda_0\rho_0}{\rho_t}+\frac{4n}{\rho_t}\left(t-\frac{\rho_1}{\rho_0}-\ldots-\frac{\rho_{t}}{\rho_{t-1}}\right),$$
which implies (after the substitution $\alpha_t:=\rho_t/\rho_{t-1}$) that
$$\lambda_\tau\leqslant 2n+4n\left(\tau-\sum_{t=1}^\tau\alpha_t\right),$$ 
and since the minimum value of $\alpha_1+\ldots+\alpha_\tau$ subject to $\alpha_t>0$ and $\alpha_1\ldots\alpha_\tau=\rho_0^{-1}$ is attained when $\alpha_1=\ldots=\alpha_\tau=\rho_0^{-1/\tau}$, we get $$\lambda_\tau\leqslant 2n+4n\tau\left(1-\rho_0^{-1/\tau}\right).$$ The right-hand side of this inequality is an increasing function of $\tau$, so it attains its maximum at the largest possible value $\tau=\log_2\rho_0$. We get $\lambda_\tau\leqslant2n+2n\log_2\rho_0$, and it remains to apply Corollary~\ref{lemPap}.
\end{proof}

The author does not expect his result to be tight even asymptotically, so this paper does not show any effort on improving the $o(n\log n)$ part of the upper bound. 

\medskip

The author is indebted to O.\,V.~Markova from Moscow State University for a series of talks on the topic, which he has had a privilege to attend since 2006. Mateusz Micha\l{}ek told the author in June 2018 about a very similar problem, known as the quantum version of Wielandt's inequality~\cite{WE}, and we quickly came to a conclusion that the progress on one of these problems can lead to the progress on the other. In particular, the author hopes that the techniques developed in this paper will allow one to get an asymptotically optimal $O(n^2)$ bound for the largest value of $\tau$ such that the equality $\F S^\tau=\Mat_n(\F)$ holds for any set $S$ for which there exists a $t$ satisfying $\F S^t=\Mat_n(\F)$ --- while the present paper gives an $O(n\log n)$ bound for the same problem but with $S^\tau$, $S^t$ replaced by $S^{\leqslant \tau}$, $S^{\leqslant t}$, respectively.

\end{document}